\documentclass[a4paper,10pt,reqno]{amsart}
\usepackage{amssymb, amsmath, graphicx, amsthm, enumitem, multicol,amsfonts}
\usepackage[margin=1in]{geometry}
\usepackage{cite}
\usepackage{ bbold }
\usepackage{longtable}
\usepackage{booktabs}
\usepackage{placeins}
\usepackage{hyperref}
\usepackage{tikz-cd}

\usepackage[utf8]{inputenc}

\DeclareMathOperator{\C}{\mathbb{C}}

\DeclareMathOperator{\Z}{\mathbb{Z}}
\DeclareMathOperator{\N}{\mathbb{N}}

\DeclareMathOperator{\cF}{\mathcal{F}}

\newcommand{\cK}{{\mathcal{K}}}
\DeclareMathOperator{\cS}{\mathcal{S}}

\newcommand{\supp}{\mathrm{supp}}

\newcommand{\sstab}{G_N}
\newcommand{\pstab}{\hat G_N}
\newcommand{\Grem}[1]{G(#1)}

\newcommand{\ba}{\begin{align}}
\newcommand{\ea}{\end{align}}
\newcommand{\bea}{\begin{eqnarray}}
\newcommand{\eea}{\end{eqnarray}}
\newcommand{\be}{\begin{equation}}
\newcommand{\ee}{\end{equation}}
\newcommand{\wt}{\mathrm{wt}}

\newcommand{\ie}{{\it i.e.~}}
\newcommand{\eg}{{\it e.g.~}}

\newcommand{\diag}{{\mathrm{diag}}}

\newcommand{\vac}{|0\rangle}
\newcommand{\idV}{\mathbf{1}_V}

\numberwithin{equation}{subsection}

\title{Limits of Vertex Algebras and Large $N$ Factorization}
\author{Thomas Gem\"unden}
\address{Thomas Gem\"unden, Independent, Bayes House, London, UK}
\email{thomas.gemunden@cantab.net}
\author{ Christoph A.~Keller}
\address{Christoph A. Keller,  Department of Mathematics, University of Arizona, Tucson, AZ 85721-0089, USA}
\email{cakeller@math.arizona.edu}

\theoremstyle{plain}
\newtheorem{thm}{Theorem}[section]
\newtheorem{lem}[thm]{Lemma}
\newtheorem{prop}[thm]{Proposition}
\newtheorem{cor}[thm]{Corollary}
\newtheorem*{theorem-non}{Theorem}

\theoremstyle{definition}
\newtheorem{defn}{Definition}[section]

\theoremstyle{remark}

\begin{document}

\begin{abstract}
We investigate the limit of sequences of vertex algebras. We discuss under what condition the vector space direct limit of such a sequence is again a vertex algebra. We then apply this framework to permutation orbifolds of vertex operator algebras and their large $N$ limit. We establish that for any nested oligomorphic permutation orbifold such a large $N$ limit exists, and we give a necessary and sufficient condition for that limit to factorize. This helps clarify the question of what VOAs are candidates for holographic conformal field theories in physics.
\end{abstract}

\maketitle

\section{Introduction}
In this article, we investigate the limit of sequences of Vertex Operator Algebras (VOA) and Vertex Algebras (VA). We address the question under what circumstances the graded vector space direct limit of such a sequence is again a VA. We find that this is decided by convergence of the structure constants. We then use this mathematical framework to address questions about existence and uniqueness of VA limits, and prove several physics conjectures about limits of permutation orbifold VOAs.

To construct a VA limit, we consider a sequence $(V^N)_{N\in\N}$ of VAs  together with connecting maps $f_{MN}: V^M\to V^N$.  We construct the vector space of the limit VA as the graded vector space direct limit of that sequence. However, we do not want to require the connecting maps to be VA homomorphisms. This means that we cannot simply work in the category of VAs and define the limit VA as the direct VA limit of the sequence; instead, we will have to construct the state-field map of the limit VA by hand.

The motivation to study limits of VOAs comes partly from physics. Let us therefore briefly explain their role in physics, and also explain why we are interested in connecting maps that are not VA homomorphisms. 
In physics, VOAs describe Conformal Field Theories (CFT) in two dimensions. They find a particularly interesting application in the AdS/CFT correspondence \cite{Maldacena:1997re,Aharony:1999ti}. This correspondence conjecturally maps theories of quantum gravity to certain types of CFTs, or, more precisely, to the limits of families of CFTs. In the original instance of the conjecture, the CFTs are described by the Lie group $SU(N)$, and the limit $N\to \infty$ is taken, so that the central charge diverges. For this reason such limits are often called \emph{large $N$ limits} or  \emph{large central charge limits}. As the central charge of two VOAs in the sequence is different, it is clear that the $f_{MN}$ are not VOA homomorphisms. Usually, they are not VA homomorphisms either. For instance, as we discuss in the second part of this article, often the limit VA has a very special property: it factorizes, even though the members of the family do not factorize. The limit is thus not homomorphic to the VAs in the sequence, and the connecting maps $f_{MN}$ are not VA homomorphisms either.

As mentioned above, our main interest is the limit of VOAs. However, it turns out that the limit VA is often no longer a VOA. For this reason, we find it most appropriate to work with \emph{grading-restricted VAs} \cite{MR3177939} rather than VOAs. 

A (grading-restricted) vertex algebra has two main ingredients: a graded vector space $V$, and a state-field map (or vertex operator map) $Y$ \cite{MR1651389,MR2023933}.
Constructing the limit VA $V^\infty$ from the family $(V^N)_{N\in\N}$ thus involves two steps. The first step is to construct its graded vector space $V^\infty$. This is relatively straightforward: We choose connecting maps $f_{MN}: V^M\to V^N$, and then define $V^\infty$ to be the direct limit of the system $(V^N,f_{MN})$. The role of the connecting map is to define how the vectors in the different $V^N$ are related. To ensure that $V^\infty$ is still grading-restricted, we need to impose some additional conditions on the connecting maps and on the vector spaces $V^N$. The most important condition is that the dimensions $\dim V^N_{(n)}$ converge as $N\to\infty$. We call a system that satisfies all these conditions a \emph{grading-restricted direct system}.  Physicists call the numbers $\dim V_{(n)}$ the \emph{spectrum}, and would therefore say that for this system ``the spectrum of $V^N$ converges".

Having constructed $V^\infty$, the second step is to construct a state-field map $Y_\infty$ on it. 
This is harder to do than the first step. As stressed above, we do \emph{not} assume that the connecting maps $f_{MN}$ are VA homomorphisms. That is, they are in general not compatible with the state field maps $Y^N$. We thus cannot simply define $Y^\infty$ as coming from a direct limit of a VA-system $(V^N,Y^N,f_{MN})$. Instead, we are forced to define the state-field map $Y^\infty$ by hand.
No surprisingly, ensuring its existence requires additional assumptions on the connecting maps. 

In our approach, we describe state-field maps $Y$ through their \emph{structure constants} $C_{abc}$,
\be
Y(b,z)c = \sum_{a\in \Phi}  z^{\wt(a)-\wt(b)-\wt(c)} C_{abc} a\ ,
\ee
that is the matrix elements of $Y(b,z)$ with respect to some homogeneous basis $\Phi$. We then define the structure constants of $Y_\infty$ as the limit of the structure constants of the $V^N$, 
\be
C^\infty_{abc}:= \lim_{N\to\infty} C^N_{abc}\ .
\ee
Crucially, we then need to ensure that the state-field map $Y_\infty$ defined in this way satisfies the axioms of a VA. Our first main result, theorem~\ref{thm:mainthm}, is that $Y_\infty$ indeed satisfies the Borcherds identity, as long as all structure constants converge: 
\begin{theorem-non}
The direct limit of a grading-restricted direct system of VAs is a grading-restricted VA, provided all structure constants converge.
\end{theorem-non}
Physicists would phrase this condition as ``all three point functions of $V^N$ converge".
The proof of this result uses the fact that we are working with grading-restricted VAs. 
It is possible to repeat our approach for VAs that are not grading-restricted, but it is then harder to establish that the state-field map exists and satisfies the VA axioms. Indeed, we give an example that has a VA-limit that is not grading-restricted in section~\ref{s:tensor}.

We mentioned above that we are primarily interested in the case where the $V^N$ are VOAs. In that case, the limit we consider is usually a large central charge limit, meaning the central charge of $V^N$ diverges as $N\to\infty$. The limit VA $V^\infty$ therefore does not contain a copy of the Virasoro algebra, and is thus not a VOA. It is however still a M\"obius VA: that is, the Lie algebra $sl(2,\C)$ generated by $L(0),L(1),L(-1)$ survives the limit. 
Moreover, if the $V^N$ are unitary, and the connecting maps are compatible with that unitary structure, then $V^\infty$ is also unitary. Such large central charge limits of unitary VOAs are probably the case of most interest in physics.

In the second part of this article we use our framework to prove certain physics conjectures; for this we focus on the special cases of permutation orbifold VAs. These conjectures have to do with existence and uniqueness of the VA limit. For instance, given a sequence of VAs, it is necessary to specify the connecting maps $f_{MN}$ in order to define its limit. However, there is a belief in physics that the choosing connecting maps does not a very important role in constructing the limit VA. On the one hand, given a family $V^N$ that has an appropriate limit as a graded vector space, it should be possible to find connecting maps which give convergent structure constants, leading to a consistent limit VA: in physics language, if the spectrum converges, then the three point functions almost automatically also converge.
On the other hand, it is also believed that if there are two different choices of connecting maps for which the limit exists, then the resulting limit VAs should be isomorphic.

Mathematically, it is clear that these  beliefs cannot hold in the generality stated above. To turn them into conjectures, we need to impose some further assumptions beyond the existence of the direct limit; physicists' belief is simply that these additional assumptions are relatively minor.  
For instance we cannot expect the structure constants $C^N_{abc}$ to converge automatically if the spectrum converges: a sequence $V^N$ that alternates between two VAs of identical spectrum but different structure constants (such as the $E8\times E8$ and the $SO(32)$ lattice VOAs) will have non-convergent structure constants, giving an immediate counterexample to the first belief. A better conjecture that actually has a chance of being true is that instead the structure constants remain bounded as $N\to\infty$. This weaker statement is usually enough for physicists, since then we can pick a convergent subsequence of $V^N$ to get a limit VA. 

We do indeed prove this weaker form of the conjecture for the case of permutation orbifolds.
In section~\ref{s:perm} we introduce permutation orbifold VAs \cite{Klemm:1990df, Dijkgraaf:1996xw, Borisov:1997nc, Bantay:1997ek}. Here, a permutation orbifold is what we call the fixed point sub-VA of an $N$-fold tensor product of a given VA under the action of a permutation group $G_N$. 
Based on previous investigations in the physics literature \cite{Lunin:2000yv,Belin:2014fna,Haehl:2014yla,Belin:2015hwa}, we established in \cite{Gemunden:2019hie} that the VA limit of permutation orbifolds exists, provided they satisfy a property we called \emph{nested oligomorphic}. That previous construction of the limit VA however was rather ad-hoc. In the language of this article, the nested oligomorphic condition guarantees that the permutation orbifolds form a grading-restricted system.

Such permutation orbifolds are an important example of large $N$ limits, and we establish several results. The first main result, theorem~\ref{thmoligo},
is that the structure constants of any nested oligomorphic permutation orbifold are bounded; by choosing suitable subsequences, it is thus always possible to find a limit VA: 
\begin{theorem-non}
Given a sequence of nested oligomorphic permutation orbifolds of grading-restricted VAs of CFT type, we can always find a subsequence that converges to a VA. 
\end{theorem-non}

The second main result, discussed in section~\ref{s:factorize}, has to do with factorization in the large $N$ limit.
We say a VA factorizes if it has a set of generators so that the commutator of the modes of any two generators only contain the identity operator. Such VAs are well-known in physics, as they allow to compute any correlation functions using Wick contractions. In particular, VOAs that appear in the AdS/CFT correspondence are expected to factorize in the large $N$ limit. In theorem~\ref{mainoligofactor} we establish a necessary and sufficient criterion for large $N$ factorization of oligomorphic permutation orbifolds involving the behavior of orbits: 
\begin{theorem-non}
The VA-limit of nested oligomorphic permutation orbifolds of grading-restricted VAs of CFT type factorizes if and only if the permutation orbifolds have no finite orbits. \end{theorem-non}

\emph{Acknowledgments:}
We thank Klaus Lux for useful discussions.  TG thanks the Department of Mathematics at University of Arizona for hospitality.
The work of TG was supported by the Swiss National Science Foundation Project Grant 175494.  The work of CAK is supported in part by the Simons Foundation
Grant No. 629215 and by NSF Grant 2111748.

\section{Large \texorpdfstring{$N$}{N} limits of vertex algebras}

\subsection{Grading-restricted vertex algebras}$ $

There are several different equivalent choices for the axioms of vertex algebras, stressing different aspects such as locality, associativity or commutativity  \cite{MR2023933}. For our purposes we find the following definition the most useful, which stresses Borcherds' identity \cite{MR843307}:

\begin{defn}\label{va}
 A {\it vertex algebra} $(V,\vac,Y)$ is a vector space $V$ with a distinguished non-zero vector $\vac$ (\emph{vacuum vector}) with a linear map $Y$ (\emph{state-field map})
 \be
 Y: V \to End(V)[[z,z^{-1}]] \qquad a \mapsto Y(a,z)=\sum_{n\in\Z}a_{n} z^{-n-1}
 \ee
 such that for all $v\in V$ $a_{(n)}v=0$ if $n$ is large enough (meaning $Y(a,z)$ is  a \emph{field}) satisfying
\begin{enumerate}
\item $Y(a,z)\vac=a+O(z)$ (\emph{creativity})
\item $Y(\vac,z)=\idV$ 
\item \emph{Borcherds' identity}: 
\begin{multline}
\sum_{j=0}^\infty \binom{m}{j}(a_{n+j}b)_{m+k-j}c = \\
\sum_{j=0}^\infty (-1)^j\binom{n}{j}a_{m+n-j}(b_{k+j}c)
- \sum_{j=0}^\infty (-1)^{j+n}\binom{n}{j}b_{n+k-j}(a_{m+j}c) \qquad \mathrm{for\ all\ } k,m,n\in \Z\ .
\end{multline}

\end{enumerate}
\end{defn}

See \eg \cite{MR1651389} for how this implies other, maybe more commonly used axioms of a VA. In particular note that when using this set of axioms, the commonly used translation operator $T$ or $L(-1)$ is defined as $Ta := a_{(-2)}\vac$.

Motivated by physics, we are actually most interested in \emph{vertex operator algebras} (VOAs) and their large central charge limit. However, for reasons that will become clear, this limit is not a VOA. It is thus more useful not to work in the framework of VOAs, but rather in the framework of \emph{grading-restricted vertex algebras} (see \eg \cite{MR3177939}:

\begin{defn}\label{grva}
 A {\it grading-restricted vertex algebra} is a vertex algebra $(V,\vac,Y)$  whose vector space $V$ is $\Z$-graded,
 \be
 V = \bigoplus_{n\in\Z} V_{(n)}\ ,
 \ee
 together with a linear map $L(0): V\to V$ defined as $L(0)v = n v$ for $v\in V_{(n)}$, satisfying
 \begin{enumerate}
\item $V_{(n)}=0$ if $n$ is sufficiently negative, and $\dim V_{(n)}<\infty$ for all $n\in\Z$ ({\it grading-restriction condition})
\item $$[L(0), Y(v, z)]=\frac{d}{dz}Y(v, z)+Y(L(0)v, z)$$
for $v\in V$. ({\it $L(0)$-bracket formula})
\end{enumerate}
\end{defn}

Note that a vertex operator algebra $(V,\vac,Y,\omega)$ is automatically a grading-restricted vertex algebra with $L(0)$ and $L(-1)$ given by the usual modes of the Virasoro field $Y(\omega,z)$.
If $a\in V_{(n)}$, we say $a$ has weight $\wt a=n$, and then $a_n$ is homogeneous of weight $-n+\wt a-1$.

\subsection{The setup}\label{ss:setup}
Let us now set up the VA limit of a family of grading-restricted vertex algebras.

\begin{defn}\label{Vdlimit}
Let $\left(V^N\right)_{N \in \N}$ be a sequence of grading-restricted vertex algebras together with a set of injective connecting maps $f_{MN}: V^M \to V^N$ for all $M \leq N$ satisfying
\begin{enumerate}
    \item $f_{NK}\circ f_{MN} = f_{MK}$ for all $M \leq N \leq K$.
    \item $f_{NN} = \mathbb{1}_{V^N}$ for all $N$.
    \item The $f_{MN}$ preserve grading and the vacuum element.
    \item \label{dimsat}
    For fixed $n \in \Z$, $\mathrm{dim}V^M_{(n)} = \mathrm{dim}V^N_{(n)}$  for all sufficiently large $M$ and $N$. 
    \item \label{dimzero} There is an $\bar n$ such that $V^N_{(n)}=0$ for $n<\bar n$ for sufficiently large $N$.
    \end{enumerate}
    We then call $(V^N,f_{MN})$ a \emph{grading-restricted direct system}.
\end{defn}
A few remarks are in order:
\begin{enumerate}
\item
For such a grading-restricted direct system, define $W^{\infty} = \bigoplus_{N \in \N} V^N$, and $\iota_N : V^N \to W^\infty$ the canonical inclusion map. Let $D \subset W^{\infty}$ be the subspace generated by elements of the form $\iota_M(u) - \iota_N\circ f_{MN}(u)$ for any $M \leq N$ and $u \in V^M$.
We then define $V^\infty$ to be the (linear) direct limit of the system $(V^N,f_{MN})$ given by
    \begin{equation}
        V^{\infty} = \varinjlim V^N = W^{\infty}/D\ .
    \end{equation}
\item
We also define $f_N$ to be the linear maps
\be
f_N:V^N \to V^{\infty}\qquad v \mapsto [\iota_N(v)]\ ,
\ee
where $[w]$ denotes the class of $w\in W^\infty$ in $V^{\infty}$. The maps $f_N$ are injective by injectivity of the $f_{MN}$ and satisfy $f_M = f_N \circ f_{MN}$ for all $M \leq N$.
\item
We insist that the connecting maps are injective. This is mainly for convenience, as it will make it easier to work with bases later on. 
\item We say the homogeneous subspace $V^M_{(n)}$ is \emph{saturated} if condition (\ref{dimsat}) holds for all $N>M$. In particular, the homogeneous components $f_{MN}^{(n)}: V^{M}_{(n)}\to V^{N}_{(n)}$ are bijective if $V^M_{(n)}$ is saturated.
\item Note that the connecting maps $f_{MN}$ do not need to be VA-homomorphisms. The system $(V^N,f_{MN})$ does therefore \emph{not} define a direct limit in the category of VAs.
    \item For more on direct limits in the context of VOAs see for example \cite{MR4380115}.

\end{enumerate}

We mention two immediate lemmas:
\begin{lem}\label{lem:urepresent}
    Any $u \in V^{\infty}$ can be written as $u=f_N(v)$ with $v \in V^N$ for some $N$.
\end{lem}
\begin{proof}
    By construction, a general element $u \in W^{\infty}$ can be written as 
    \begin{equation}
        u = \sum_{i=1}^I \iota_{M_i}(v^i), 
    \end{equation}
    where $u^{i} \in V^{M_i}$ for a set of integers $\{M_i\}$. This means that $[u]\in V^{\infty}$ can be written as $\sum_{i=1}^I f_{M_i}(v^i)$.
    Taking $N = \max\{M_i\}$,     
    \begin{equation}
        u  = \sum_{i=1}^I f_{N}(f_{M_i N }(v^i)) = f_{N} \left[\sum_{i=1}^I f_{M_i N}(v^i)\right].
    \end{equation}
\end{proof}
 We will often write $u = f_N(u^N) =: u^N$ and suppress the $f_N$, where equality in $V^{\infty}$ is understood, call $u^N$ the representative of $u$ in $V^N$. It is unique in $V^N$ because $f_N$ is injective.
 
 \begin{lem}
    $V^{\infty}$ is graded by weights with finite-dimensional homogeneous subspaces.
\end{lem}
\begin{proof}
    $V^\infty$ is graded because $W^\infty$ is graded and the $f_{MN}$ preserve the grading, so that quotienting by $D$ preserves the grading. 
    Let $\Phi^M:=\bigcup_{n} \Phi^M_n$ be a homogeneous basis of $V^M$, that is $\Phi^M_n$ a basis for $V^M_{(n)}$. For a fixed $n$, let $M$ be such that $V^M_{(n)}$ is saturated. 
    By lemma~\ref{lem:urepresent}, any vector $u\in V^\infty_{(n)}$ can be written as $f_N(u^N)$ for some vector $u^N$, where we can take $N\geq M$. Because the homogeneous components $f^{(n)}_{MN}$ are bijective, $u^N$ can be expressed as a linear combination of the vectors in $f^{(n)}_{MN}(\Phi^M_n)$, meaning that $u$ can be expressed as a linear combination in $f_M(\Phi^M_n)$. Moreover, since the $f_M$ are injective, the vectors $f_M(\Phi^j_n)$ are linearly independent, so that $\Phi_n:=f_M(\Phi^M_n)$ is indeed a finite basis of $V^\infty_{(n)}$.
\end{proof}
It follows that 
\be
\Phi:=\bigcup_{n} \Phi_n
\ee 
with the $\Phi_n$ constructed as above is a homogeneous basis of $V^\infty$. We will frequently use this basis in what follows.

\subsection{The restricted dual}
It will be useful to work with the dual space of limit VAs.
The dual space $(V^\infty)^*$ of $V^\infty$ itself is given by the inverse limit of the system $(V^N,f_{MN})$. For completeness, let us give the standard definition and properties of this construction.

 Let $\left(V^N\right)^*$ be the dual of $V^N$. Define the surjective, dual connecting maps (bonding maps) $f'_{MN}:\left(V^N\right)^* \to \left(V^M\right)^* $ for all $M \leq N$ by 
\begin{equation}
    \langle f'_{MN}(v'), u \rangle = \langle v', f_{MN}(u) \rangle,
\end{equation}
for all $v' \in \left(V^N\right)^*$ and $u \in V^M$, where we introduce the evaluation map $\langle v', u\rangle :=v'(u)$.
Then for all $M \leq N \leq K$ the following relations hold
\begin{equation}
    f'_{MN} \circ f'_{NK} = f'_{MK}.
\end{equation}

\begin{defn}
 The inverse limit of the duals is defined by
 \begin{equation}
     \varprojlim \left(V^N\right)^* = \{ v' \in \prod_{M = 1}^{\infty}\left(V^M\right)^*|v'_M = f'_{MN}(v'_N) \text{ for all } M \leq N  \}.
 \end{equation}
 For every $M$, there exists a canonical surjective map \begin{equation}
     \pi_M:\varprojlim \left(V^N\right)^* \to \left(V^M\right)^*, 
     v' \mapsto v'_M,
 \end{equation}
 such that 
 \begin{equation}
     \pi_M = f'_{MN}\circ \pi_N,
 \end{equation}
 for all $M \leq N$. The entry $v'_M$ is therefore the representative of $v'$ in $(V^M)^*$.
\end{defn}
There is a canonical (linear) isomorphism $\left(V^{\infty}\right)^* \cong \varprojlim \left(V^N\right)^*$, with the canonical pairing given by
\begin{equation}
\langle v', u \rangle =  \langle v'_N , u^N \rangle\ .  
\end{equation}
Note this definition does not depend on the choice of representative $u^N$ of $u$ since
\begin{equation}
    \langle v'_N , u^N \rangle = \langle v'_M , f_{MN}(u^M) \rangle = \langle f'_{MN}(v'_N) , u^M \rangle = \langle v'_M , u^M \rangle.
\end{equation}

When working with VOAs (or grading-restricted VAs), it is better not to work with the full dual space $V^*$, but rather the restricted graded dual $V'$ of $V$, defined as
 \be
 V':=\bigoplus V_{(n)}^* \subset V^*
 \ee
For $V^\infty$ it is given by 
 \be
 (V^\infty)':=\bigoplus(V^\infty_{(n)})^*\ .
 \ee
We can characterize it by restricting to $v'\in (V^\infty)^*$ such that
\be
v' |_{V^{\infty}_{(n)}}=0 \qquad \textrm{for almost all } n\ .
\ee

\subsection{Vertex operators and matrix elements}
In section~\ref{ss:setup}, we defined $V^\infty = \varinjlim V^N$ as a grading-restricted direct limit.
We now want to define the state-field map $Y_\infty$ on $V^{\infty}$ through its matrix elements:

\begin{defn}\label{matels}
Assuming that the limit exists,
we define the matrix elements of $Y_\infty$ on $V^\infty$ as
 \begin{equation}\label{Ymatdef}
     \langle v', Y_{\infty}(u,z)w\rangle = \lim_{N \to \infty} \langle v'_N,  Y_N(f_{MN}(u^M),z)f_{KN}(w^K) \rangle,
 \end{equation}
 for all $u,w \in V^{\infty}$ and $\vec{v'} \in \left(V^{\infty}\right)'$. Here $u^M \in V^M$ and $w^K \in V^K$ are representatives of $u$ and $w$, and we assume $N \geq M,K$. 
\end{defn}

A few remarks:
\begin{enumerate}
\item
This definition is independent of the choice of $M,K$: Choosing $M'\geq M$, for example, we find that 
\begin{equation}
\begin{split}
        \lim_{N \to \infty} \langle v'_N,  Y_N(f_{M'N}(u^{M'}),z)f_{KN}(w^K) \rangle & =   
        \lim_{N \to \infty} \langle v'_N,  Y_N(f_{M'N}(f_{M M'}(u^M)),z)f_{KN}(w^K) \rangle \\
       & = \lim_{N \to \infty} \langle v'_N,  Y_N(f_{MN}(u^M),z)f_{KN}(w^K) \rangle.
        \end{split}
\end{equation}
\item
The limit in (\ref{Ymatdef}) is the ordinary limit in $\C$ order by order in the formal power series in $z$. Equivalently, we can write $Y_\infty(u,z)x$ as a limit in the restricted weak topology, that is the weak topology with respect to the restricted dual $V'$,
\begin{equation}
    Y_{\infty}(u,z)w = \lim_{N \to \infty} f_N\left(Y_N(u^N,z)w^N\right)\ .
\end{equation}
\end{enumerate}

We now use this to define the structure constants and the state field map on $V^\infty$.
Let $\Phi:=\bigcup_{n} \Phi_n$ be a homogeneous basis of $V^\infty$. Because all $V^\infty_{(n)}$ are finite dimensional, we can pick a homogeneous dual basis $\Phi'$, so that for $b\in \Phi, a \in \Phi'$
\be
\langle a, b\rangle = \delta_{a,b}\ .
\ee
For convenience we will simply identify $\Phi$ and $\Phi'$ and their vectors.
Since $a,b,c$ are homogeneous, $ \langle a, Y_{\infty}(b,z)c\rangle= \langle a, b_{\wt(c)+\wt(b)-\wt(a)-1} c \rangle 
z^{\wt(a)-\wt(b)-\wt(c)}$. This leads to the following definition:

\begin{defn}\label{strcst}
Let $(V^N,f_{MN})$ be a grading-restricted direct system, and $\Phi$ a homogeneous basis of $V^\infty$ as above. For $a,b,c\in \Phi$ with representatives $a^N,b^N,c^N$, define the \emph{structure constants} 
\be
C^N_{abc} := \langle a^N, b^N_{\wt(c^N)+\wt(b^N)-\wt(a^N)-1} c^N \rangle\ ,
\ee
and
\be
C^\infty_{abc}:= \lim_{N\to\infty} C^N_{abc}\ .
\ee
Assuming the limit in definition~\ref{strcst} exists for all $a,b,c\in \Phi$,
we define the map $Y_\infty: V^\infty\to End(V^\infty)[[z,z^{-1}]]$ as
\be
Y_\infty(b,z)c := \sum_{a\in \Phi}  z^{\wt(a)-\wt(b)-\wt(c)} C^\infty_{abc} a
\ee
\end{defn}
Note that this definition is compatible with definition~\ref{matels}. Also note that for convenience of notation, we used the dual basis in the definition of $C_{abc}$; instead, we could of course have considered $Y(b,z)c$ and read off the coefficient of $a$, as $Y(b,z)c = \sum_{a\in \Phi}  z^{\wt(a)-\wt(b)-\wt(c)} C_{abc} a$.

\subsection{The VA-limit}
Let us now establish the first main result of our paper:
\begin{thm}\label{thm:mainthm}
Let $(V^N,f_{MN})$ be grading-restricted direct system with limit $V^\infty$,  $\vac := f_N(\vac)$ for some $N$, and assume that the limit of all the structure constants as in definition~\ref{strcst} exists. Then $(V^\infty,\vac,Y_\infty)$ is a grading-restricted vertex algebra, the  \emph{(grading-restricted) VA-limit} of the system $(V^N,f_{MN})$.
\end{thm}

\begin{proof}
First we note that $V^\infty$ satisfies the grading-restriction condition due to (\ref{dimsat}) and (\ref{dimzero}) of definition~\ref{Vdlimit}.
This immediately implies that
$Y_\infty(v,z)$ is a field for any $v$ because $V_\infty$ is grading-restricted, \ie $\wt(a)$ is bounded from below.

Next, 
\be
Y_\infty(u,z)\vac = \lim_{N\to\infty}f_N (Y_N(u^N,z)\vac) = \lim_{N\to\infty}f_N(u^N+O(z))= u+O(z)
\ee
establishes creativity, and $\lim_{N\to\infty}f_N(Y_N(\vac,z)u^N)=u$ establishes $Y_\infty(\vac,z)=1_{V^\infty}$.

To show the $L(0)$ bracket formula, note that the $L(0)$-operator commutes with the connecting maps in the sense that
\begin{equation}
    f_{MN}(L_M(0)v) = L_N(0)f_{MN}(v),
\end{equation}
and that $L(0)v = \lim_{N\to\infty} f_N(L_N(0)v^N)$.
It then follows 
\begin{equation}
     \begin{split}
     [L(0),Y_{\infty}(v,z)]u & = \lim_{N \to \infty} f_N([L_N(0),Y_{N}(v^N,z)]u^N) \\
     & = \lim_{N \to \infty}f_N( \frac{d}{dz}Y_N(v^N,z)u^N + Y_N(L_N(0)v^N,z)u^N \\
     & = \frac{d}{dz}Y_{\infty}(v,z)u + Y_{\infty}(L(0)v,z)u,
     \end{split}
\end{equation}
where we can exchange limit and the formal derivative term by term in the formal power series.

Finally, let us prove that $Y_\infty$ satisfies Borcherds' identity. Let $\Phi$ be a homogeneous basis of $V^\infty$. By \cite{Gemunden:2019hie}, Borcherds' identity is satisfied if the following condition on the structure constants holds:
Defining 
\be
j_1=\wt(b)+\wt(a)-\wt(d)-n-1\ , \ j_2=\wt(c)+\wt(b)-\wt(d)-k-1\ , \ j_3=\wt(c)+\wt(a)-\wt(d)-m-1\ ,
\ee
for all $a,b,c,e \in \Phi$
\be\label{3ptJacobi}
\sum_{d\in\Phi} \binom{m}{j_1} C^\infty_{edc}C^\infty_{dab} =
\sum_{d\in\Phi} (-1)^{j_2}\binom{n}{j_2} C^\infty_{ead}C^\infty_{dbc} - \sum_{d\in\Phi} (-1)^{j_3+n}\binom{n}{j_3} C^\infty_{ebd}C^\infty_{dac}
\ee
holds for all $k, n,m\in \Z$ such that $j_1,j_2,j_3\geq 0$.

Note that due to the condition $j_i\geq 0$ and the grading-restriction condition of $V^\infty$, for fixed $k,n,m$ only finitely many weights $\wt(d)$ contribute. Since the homogeneous subspaces $V^\infty_{(n)}$ are finite dimensional, the sum over $d\in \Phi$ has only finitely many terms.

We can thus find an $N$ that gives representatives $a^N,b^N,c^N,e^N,d^N$. If $N$ is large enough, then the $d^N$ form a basis for the relevant homogeneous subspaces $V^N_{(n)}$. 
The analog of identity (\ref{3ptJacobi}) with structure constants $C^N_{edc}$ is then automatically satisfied for all $N$, since the $Y_N$ are vertex operator maps. We can thus take the limit $N\to\infty$ of those identities, and exchange the limit with the finite sum over $d\in \Phi$ to establish that (\ref{3ptJacobi}) holds for the structure constants $C^\infty_{edc}$.

\end{proof}

\section{A (non-)example: Tensor Power VAs}\label{s:tensor}
Let us now construct some examples of such limits. Our first example is in a sense a non-example: the assumptions of definition~\ref{Vdlimit} are not satisfied, and even though the VA limit exists, it turns out to be a VA that violates the grading restriction axiom, and is hence not a grading-restricted VA. However, the example will serve as a useful starting point for the permutation orbifolds discussed in the next section.

\subsection{Seed VAs and tensor products}\label{ss:seedtensor}
Let $V$ be a grading-restricted VA that is of the form \be\label{CFTtype}
V=\C\vac\oplus \bigoplus_{n=1}^\infty V_{(n)}\ . 
\ee
In the context of VOAs, this is usually called a VOA of \emph{CFT type}. We call $V$ the \emph{seed VA}.
For future use we note that for such a VA, unless  $a\in \C\vac$,
\be\label{1ptfctvanish}
C_{a\vac\vac}=C_{\vac a\vac}=C_{\vac\vac a}=0\ ,
\ee
because of the identity and creativity properties.

Now let us consider tensor products of the seed VA $V$. Let $I_N:=\{1,2,\ldots, i_N\}$ be the set of the first $i_N$ numbers.
Denote by $V^{\otimes |I_N|}$ the $|I_N|$-th tensor power of the seed VA, with the grading given by the sum of the gradings of the individual factors. Clearly this is again a grading-restricted VA of the form (\ref{CFTtype}).

For future use, it will be useful to describe a basis of this tensor product VA in the following way: Let $\Psi$ be a homogenenous basis of $V$ with $\vac$ the basis vector for $V_{(0)}$, and $a$ be a function $I_N \to \Psi$. We define the weight of $a$ as 
\be
|a|:=\sum_{i \in I_N} \wt(a(i))\ ,
\ee
and its  support as 
\be
\supp (a) := \{ i\in I_N : a(i)\neq \vac\}\ .
\ee
Such a function $a$ defines a vector in $V^{\otimes |I_N|}$ by
\be\label{fcttovec}
a = \bigotimes_{i\in I_N}a(i)\ ,
\ee
which by abuse of notation we denote by the same symbol $a$.
Let $\cF^N_n$ be the set of all such functions $a: I_N\to\Psi$ of weight $n$. It is then clear that (by the same abuse of notation) $\cF^N = \bigcup_n \cF^N_n$ forms a homogeneous basis of $V^{\otimes |I_N|}$. 
Next, for a tensor product state $v=\bigotimes_{i \in I_N} v_i \in V^{\otimes |I_N|}$, we define its support
\be
\cK_v = \supp(v) := \{ i \in i_N: v_i \notin V_{(0)}\}\ .
\ee
For vectors that come from functions as in (\ref{fcttovec}), the two definitions of support of course agree.
Since $V^{\otimes N}$ is a direct sum of vector spaces of definite support, in the future we will mostly work with states of definite support, and extend our results by linearity if needed.

\subsection{Connecting maps}
Now assume that $|I_N|$ is monotonically growing in $N$.
For all $M \leq N$ we define the connecting maps 
\begin{equation}
    \begin{split}
        g_{MN}: & V^{\otimes |I_M|} \to  V^{\otimes |I_N|} \\
        & v_1 \otimes \ldots v_{|I_M|} \mapsto  v_1 \otimes \ldots v_{|I_M|} \otimes \underbrace{\vac \otimes\ldots\otimes \vac}_{|I_N|-|I_M|}.
    \end{split}
\end{equation}
$(V^{\otimes |I_N|},g_{MN})$ is then not quite a grading-restricted direct system: It is clear that conditions (1)--(3) of definition~\ref{Vdlimit} are satisfied, and by virtue of (\ref{CFTtype}), so is (5). Unless $V$ is trivial however, the dimensions of most $V^{\otimes |I_N|}_{(n\geq 1)}$ will diverge, so that (4) is violated.  However, we will now prove that it is still possible to define a VA structure on $V^\infty$.

\subsection{Duals} ${ }$
We can construct the dual spaces in the same way as before: the duals of the components are
\begin{equation}
\left( V^{\otimes |I_N|} \right)^* = \left( V^* \right)^{\otimes |I_N|},
\end{equation}
where the canonical pairing is given by 
\begin{equation}
    \langle v'_1 \otimes \ldots \otimes v'_{|I_N|}, v_1 \otimes \ldots \otimes v_{|I_N|} \rangle = \prod_{i = 1}^{|I_N|} \langle v'_i,v_i \rangle.
\end{equation}
As before, we construct the dual $(V^{\otimes\infty})^*$ as the inverse limit, and the restricted dual $(V^{\otimes\infty})'$ as the appropriate subset.

For what follows however it is useful to construct an appropriate decomposition of $(V^{\otimes \infty})$. 
Let $\vac' \in V'_0$ be the unique functional satisfying $\langle \vac', \vac \rangle = 1$  and define connecting maps on the dual spaces for all $M \leq N$ by
\begin{equation}
    \begin{split}
        \tilde g_{MN}: & \left(V^{\otimes |I_M|}\right)' \to  \left(V^{\otimes |I_N|}\right)' \\
        & v'_1 \otimes \ldots \otimes v'_{|I_M|} \mapsto  v'_1 \otimes \ldots \otimes v'_{|I_M|} \otimes \underbrace{\vac' \ldots \vac'}_{|I_N| - |I_M|}.
    \end{split}
\end{equation}
These are again connecting maps, so that we can take the direct limit of the system $(\left(V^{\otimes |I_N|}\right)^*,\tilde g_{MN})$.
Note that $\varinjlim \left(V^{\otimes i}\right)^* \subset \left(V^{\otimes \infty}\right)^* = \varprojlim \left(V^{\otimes i}\right)^*$.
This gives injective maps 
\begin{equation}
    \begin{split}
        \tilde g_{N}: & \left(V^{\otimes |I_N|}\right)' \to  \varinjlim \left(V^{\otimes |I_N|}\right)^* \\
        & v'_1 \otimes \ldots \otimes v'_{|I_N|} \mapsto  [v'_1 \otimes \ldots \otimes v'_{|I_N|}].
    \end{split}
    \end{equation}

We can use this to find an orthogonal decomposition of $\left(V^{\otimes \infty}\right)'$: 
\begin{lem}
For all $N\in \N$,
 \begin{equation}\label{Vtensrest}
     \left(V^{\otimes \infty}\right)' = Im(\tilde g_N) \bigoplus Im(g_N)^{\perp}\ ,
 \end{equation}
 where $Im(g_N)^{\perp}$ is the annihilator of $Im(g_N)$.
\end{lem}
\begin{proof}
Consider $(V^{\otimes\infty}_{(n)})^*$. We know that for all $N$, $Im(\tilde g_N)_{(n)}\subset (V^{\otimes\infty}_{(n)})^*$ is finite dimensional,
and satisfies
\be
Im(g_N)_{(n)}\cong Im(\tilde g_N)_{(n)}\ .
\ee
Pick a basis $\{e'_i\}$ of $Im(\tilde g_N)_{(n)}$ and a dual basis $\{e_j\}$ of $Im(g_N)_{(n)}$. Then define $P(v):= \sum_i \langle v,e_i\rangle e_i'$. In the usual way, $v=P(v) + (v-P(v))$ then gives the decomposition $(V^{\otimes\infty}_{(n)})^*=Im(\tilde g_N)_{(n)}\bigoplus Im(g_N)^\perp_{(n)}$, from which (\ref{Vtensrest}) follows. 
\end{proof}

\subsection{The Vertex Algebra \texorpdfstring{$V^{\otimes \infty}$}{V infinity}}

\begin{thm}
The limit $(V^{\otimes\infty},\vac,Y_\infty)$ of the system $(V^N,g_{MN})$ exists and is a vertex algebra.
\end{thm}
\begin{proof}
Note that we cannot directly apply theorem~\ref{thm:mainthm} because, as pointed out above, the grading restriction condition is not satisfied, and in general $\dim V^\infty_{(n)}=\infty$. First, we show that $Y_\infty$ is well-defined.
   Let $M$ be such that $u,w \in Im(g_M)$ and $v' \in \left(V^{\otimes \infty}\right)'$. The vertex operators are defined as
   \begin{equation} 
       \langle v', Y_{\infty}(u,z)w\rangle = \lim_{N \to \infty} \langle v'_N, Y_N(u^N,z)w^N \rangle
   \end{equation}
   We first note that for $u,v\in V^{\otimes|I_M|}$,
  \begin{multline}
  Y_N(g_{MN}(u),z)g_{MN}(v)=Y(u_1,z)v_1 \otimes \ldots \otimes Y(u_{|I_M|},z)v_{|I_M|} \otimes (\vac)^{\otimes(|I_N|-|I_M|)} \rangle\\ = g_{MN} (Y_M(u,z)v) \rangle
  \end{multline}
   Now we use the decomposition (\ref{Vtensrest}):
For $v' \in Im(g_M)^{\perp}$,  
   \begin{equation}
       \langle v', Y_{\infty}(u,z)x\rangle = \lim_{N \to \infty} \langle v'_N, Y_N(g_{MN}(u^M),z)g_{MN}(w^M) \rangle = 
       \lim_{N \to \infty} \langle v'_N, Y_N(u^N,z)w^N \rangle=
       0,
   \end{equation}
   while for $v' \in Im(\tilde g_M)$, we find that
      \begin{equation} 
   \begin{split}
       \langle v', Y_{\infty}(u,z)w\rangle & = \lim_{N \to \infty} \langle v'_N, Y_N(u^N,z)w^N \rangle \\
       & = \lim_{N \to \infty} \langle v'_1 \otimes \ldots \otimes v'_{|I_M|} \otimes (\vac')^{\otimes (|I_N|-|I_M|)}, Y(u_1,z)w_1 \otimes \ldots \otimes Y(u_{|I_M|},z)w_{|I_M|} \otimes (\vac)^{\otimes (|I_N|-|I_M|)} \rangle \\
       & = \langle v'_M, Y(u^M,z)w^M \rangle.
    \end{split}
   \end{equation}

   Hence the structure constants $C_{abc}$ converge and $Y_\infty$ exists. Moreover, $Y_\infty$ is a state-field map since $\wt(a)\geq 0$ for all vectors $a$. Creativity, identity and the $L(0)$ bracket formula follow by the same argument as in theorem~\ref{thm:mainthm}.
     
   Finally, Borcherds' identity follows by a similar argument as in theorem~\ref{thm:mainthm}. The complication here is that the homogeneous components $V^\infty_{(n)}$ are no longer finite dimensional, so that we can no longer find a $N$ such that the $d^N$ are representatives of a basis of $V^\infty_{(n)}$. However, note that by construction of the tensor product that if $a,b$ have representatives $a^N,b^N$, then $b_{(n)} a \in Im(g_N)$, which has finite dimensional homogeneous components. We can thus pick a finite basis $d^N$ of $Im(g_N)$ and insert it in the same way as in theorem~\ref{thm:mainthm} to establish Borcherds' identity.
\end{proof}

We note that in this specific example, the connecting maps $g_{MN}$ actually are VA homomorphisms. Therefore, we could have constructed $V^\infty$ as the direct limit of (not grading-restricted) VAs, without worrying about convergence of the structure constants. If the seed VAs are VOAs, then the $g_{MN}$ are still not VOA homomorphisms, since the conformal vector does not get mapped to the conformal vector.

\subsection{An action by \texorpdfstring{$S_{\infty}$}{S infinity}}\label{ss:s_infty}
We established that tensor product VAs have a large $N$ limit. This limit however is not a grading-restricted VA. For this reason, we want to investigate limits of permutation orbifolds instead.

For ease of notation, let us take $I_N=\{1,2,\ldots N\}$ for the moment.
Any permutation $\sigma \in S_N$ acts naturally as an VA automorphism on $V^{\otimes N}$ by 
\begin{equation}
    \sigma \cdot v_1 \otimes \ldots \otimes v_N = v_{\sigma^{-1}(1)} \otimes \ldots \otimes v_{\sigma^{-1}(N)} 
\end{equation}
The symmetric group $S_N$ is thus a group of automorphisms of $V^{\otimes N}$.

For all $M \leq N$, define connecting maps  $\phi_{MN}: S_M \to S_N$ by mapping a permutation $\sigma \in S_M$ to the corresponding permutation $\sigma^N \in S_N$, that acts trivially on the last $N-M$ objects.
Clearly, the $\phi_{MN}$ satisfy 
\begin{align}
    & \phi_{NK}\circ \phi_{MN} = \phi_{MK} \text{ for all } M \leq N \leq K \\
    & \phi_{NN} = 1 \text{ for all } N.
\end{align}
Furthermore, they are compatible with the connecting maps $g_{MN}$ in the sense that
\begin{equation}\label{eq:perm_comp}
    g_{MN}(\sigma \cdot v) = \phi_{MN}(\sigma) \cdot g_{MN}(v),
\end{equation}
for all $v \in V^{\otimes M}$ and $\sigma \in S_M$.

In view of the above, it is tempting to try to define limits of symmetric orbifold VAs in the following way: Define $S_{\infty}$ as the direct group limit of $S_N$ under the connecting maps $\phi_{MN}$, $S_\infty = \coprod_N S_N / \sim_{\phi_{MN}}$.
 Then $S_{\infty}$ acts on $V^{\otimes \infty}$ by 
\begin{equation}
    \sigma \cdot u = g_N(\sigma^N \cdot u^N).
\end{equation}
This definition is independent of the choice of representatives.
Then $S_{\infty}$ is indeed a group of automorphisms of $V^{\infty}$:
\begin{equation}
    \begin{split}
        \sigma \cdot \left(Y_{\infty}(u,z)v\right) & = \lim_{N \to \infty} \sigma \cdot g_N\left(Y_N(u^N,z)v^N\right) \\
        & = \lim_{N \to \infty} g_N\left(\sigma^N \cdot (Y_N(u^N,z)v^N)\right) \\
        & = \lim_{N \to \infty} g_N\left(Y_N(\sigma^N \cdot u^N,z)\sigma^N \cdot v^N\right) \\
        & = \lim_{N \to \infty} g_N\left(Y_N((\sigma \cdot u)^N,z)(\sigma^N \cdot v^N)^N\right) \\
        & = Y_{\infty}(\sigma \cdot u,z) \sigma \cdot v\ .
    \end{split}
\end{equation}
Since $S_\infty$ is an automorphism of $V^\infty$, we can now in principle consider the fixed-point VA $\left(V^{\otimes \infty}\right)^{S_{\infty}}$. However, this does not lead to an interesting result, since $\left(V^{\otimes \infty}\right)^{S_{\infty}}$ is trivial: By lemma~\ref{lem:urepresent}, any vector $v$ in $V^\infty$ will be in $Im(g_N)$ for some $N$. To be invariant under all transpositions $(M,N+1)$ with $1\leq M \leq N$, $v$ has to be the vacuum vector. For this reason we need to take a different approach to limits of permutation orbifolds.

\section{An example: Permutation orbifolds}\label{s:perm}

\subsection{Connecting maps for the permutation orbifolds}
In subsection \ref{ss:s_infty}, we constructed an action of the permutation group $S_{|I_N|}$ on the vertex algebra $V^{\otimes |I_N|}$. Let us now consider the action of a permutation group $G_N \leq S_{|I_N|}$.
Define the projector from $V^{\otimes|I_N|}$ onto $V^N := \left(V^{\otimes |I_N|}\right)^{G_N}$ as
\begin{equation}
    \pi_N = \frac{1}{|G_N|} \sum_{\sigma \in G_N} \sigma.
\end{equation}
To construct a basis of $V^N$, we start with the basis $\cF^N$ of $V^{\otimes|I_N|}$ defined in section~\ref{ss:seedtensor}. Note that $G_N$ acts on $a\in \cF^N$ by $\sigma\circ a(i)=a( \sigma^{-1}i)$. Clearly, a homogeneous basis for $V^N_{(n)}$ is then given by
\be
\Phi^N = \bigcup_{n\in\N} \Phi^N_n\ , \qquad \Phi^N_n = \pi_N (\cF^N_n)\ .
\ee
The number of elements in this basis is given by the number of orbits of functions of weight $n$ under $G_N$, which we denote by $b_n(G_N)$,
\be
|\Phi^N_n| = b_n(G_N)\ .
\ee
Next, we want to define connecting maps. For this, we introduce some notation, following \cite{Gemunden:2019hie}:
\begin{defn}\label{def:oligo}
Let $\cK \subset I_N$.
\begin{enumerate}
\item Denote by $\sstab^{\cK} := \{\sigma \in G_N|k \sigma \in \cK, \forall k \in \cK\}$ the setwise stabilizer of $\cK$. 
\item Denote by $\pstab^{\cK} := \{\sigma \in G_N|k \sigma = k, \forall k \in \cK\}$ the pointwise stabilizer of $\cK$. 
\item Let $\Grem{\cK}^N$ be the permutation group defined by the action of 
$\sstab^{\cK}/\pstab^{\cK}$ on $\cK$. Note that $\Grem{\cK}^N$ is the restriction of $G_N$ to $\cK$ in the natural sense. \label{Gremdef}
\end{enumerate}
\end{defn}
Note that $\pstab^{\cK}$ is a normal subgroup of $\sstab^{\cK}$, so that definition (\ref{Gremdef}) makes sense.

We now construct the connecting maps $f_{MN}$ recursively:
\begin{defn}\label{permconnectmap}
Define the linear maps $\bar f_{MN}: V^{\otimes |I_M|}\to V^N$ in the following way:
\be
\bar f_{NN} = \pi_N\ 
\ee
For $v\in V^{\otimes |I_N|}$ with definite support $\supp(v)=:\cK_v$,
\be\label{fbardef}
\bar f_{N,N+1}(v) = \sqrt{\frac{|G_N||\pstab^{\cK_v}|}{|G_{N+1}||\hat G_{N+1}^{\cK_{g_{N,N+1}(v)}}|}} \pi_{N+1} \circ g_{N,N+1}(v)
\ee
and $f_{MN}= f_{N-1,N}  \circ \cdots \circ f_{M+1,M+2}\circ f_{M,M+1}$ for $N>M$.
We then define the connecting maps $f_{MN}: V^M\to V^N$ as
\be
f_{MN}:= \bar f_{MN}|_{V^M}\ .
\ee
\end{defn}
As we will see, the unwieldy prefactor in (\ref{fbardef}) is necessary for the structure constants to converge. Below we will give a much nicer expression for the homogeneous components $f^{(n)}_{NM}$ in the case when $V^N_{(n)}$ is saturated. But first, we need to impose some additional conditions on the family $G_N$ to ensure that the system $(V^N,f_{MN})$ is indeed a grading-restricted direct system as in definition~\ref{Vdlimit}.

\subsection{Nested oligomorphic permutation orbifolds}

Following \cite{Gemunden:2019hie}, we make the following definition:
\begin{defn}\label{conv}
Assume $|I_N|<|I_{N+1}|$.
Let the family of permutation groups $(G_N)_{N\in\N}$ satisfy the conditions:
\begin{enumerate}
    \item The numbers $b_n(G_N)$ converge for all $n$.\label{condoligo}
    \item For every finite set $\cK\subset\N$, there is a group $\Grem{\cK}$ such that $\Grem{\cK}^N = \Grem{\cK}$ for $N$ large enough.\label{condGK}
    \item $\Grem{I_{N-1}}^N< G_{N-1}$ for all $N$.\label{condnest}
\end{enumerate}
We then call $G_N$ \emph{nested oligomorphic}.
\end{defn}

\begin{prop}\label{permGRDS}
Let $V$ be a seed VA as in (\ref{CFTtype}),
$G_N$ a nested oligomorphic family of permutation groups, $V^N = \left(V^{\otimes |I_N|}\right)^{G_N}$ and $f_{MN}$ as in definition~\ref{permconnectmap}. Then $(V^N,f_{MN})$ is a grading-restricted direct system. 
\end{prop}
\begin{proof}
(1)--(3) of definition~\ref{Vdlimit} follow immediately by construction of the $f_{MN}$, and (5) from the form of $V^N$.
The nesting condition (\ref{condnest}) in definition~\ref{conv} implies that if two elements of $V^{\otimes|I_{N-1}|}$ are in different orbits under $G_{N-1}$, then they are also in different orbits as elements of $V^{\otimes|I_N|}$ under $G_N$; hence $f_{N-1,N}$ is injective.
Finally, (4) follows from (\ref{condoligo}) in definition~\ref{conv}.
\end{proof}

\begin{lem}\label{satpro}
Fix $n$. If $V^M_{(n)}$ is saturated at $M$ and $N\geq M$, then
\be
\pi_N \circ g^{(n)}_{MN}\circ \pi_M = \pi_N \circ g^{(n)}_{MN}\ .
\ee
\end{lem}
\begin{proof}
Saturated means that $b_n(G_M)=b_n(G_N)$. As in the proof of proposition~\ref{permGRDS}, the (iterated) nesting condition (\ref{condnest}) implies that different $G_M$ orbits are in different $G_N$ orbits, which together with saturation implies that the $G_M$ orbits are in one-to-one correspondence to $G_N$ orbits. This means that if we pick a representative $v$ of a $G_M$ orbit, then for any $\sigma \in G_M$, we can find $\tau_\sigma \in G_N$ such that $\tau_\sigma$ acts on $g_{MN}(v)$ as $\phi_{MN}(\sigma)\in S_N$,
\be
\tau_\sigma\circ g_{MN}(v) = \phi_{MN}(\sigma)\circ g_{MN}(v)\ . 
\ee 
Using this, we can write 
\begin{align}
    \pi_N \circ g^{(n)}_{MN}\circ \pi_M(v) & = \frac{1}{|G_M|} \sum_{\sigma \in G_M} \pi_N \circ g^{(n)}_{MN} \circ \sigma(v)  \\
    & = \frac{1}{|G_M|} \sum_{\sigma \in G_M} \pi_N\circ \phi_{MN}(\sigma)\circ g^{(n)}_{MN}(v)  \\
    & = \frac{1}{|G_M|} \sum_{\sigma \in G_M} \pi_N\circ \tau_{\sigma}\circ g^{(n)}_{MN}(v)  \\
    & = \pi_N \circ g^{(n)}_{MN}(v)\ .
\end{align}
\end{proof}
The lemma immediately implies
\begin{cor}
Fix $n$. If $V^N_{(n)}$ is saturated at $N$ and $M\geq N$, then
\be
f^{(n)}_{NM}= \sqrt{\frac{|G_N||\pstab^{\cK_v}|}{|G_{M}||\hat G_{M}^{\cK_{g_{NM}(v)}}|}} \pi_M \circ g^{(n)}_{NM}\ .
\ee
\end{cor}
\begin{proof}
Use lemma~\ref{satpro} and induction in $M$.
\end{proof}

\subsection{Structure constants}\label{ss:permstruc}
Now let us investigate convergence of the structure constants for nested oligomorphic permutation orbifolds. 
Let $a^N,b^N,c^N \in \cF^N$ be basis vectors of $V^{\otimes|I_N|}$.
By the remarks above, $\cF^N$ forms a basis  of $V^{\otimes|I_N|}$. The structure constants of the tensor VA $V^{\otimes|I_N|}$ are given by
\be\label{Cg1g2g3}
c^N(a^N,b^N,c^N):= \prod_{i=1}^{|I_N|}c(a^N(i),b^N(i),c^N(i))
\ee
where for all $i \in I_N$, $a^N(i),b^N(i),c^N(i)\in \Psi$ are basis elements of the seed VA $V$, and 
\be
c(a^N(i),b^N(i),c^N(i)) 
\ee
are the structure constants of the seed VA $V$.
By the remarks above, $\Phi^N_n = \pi_N (\cF^N_n)$ is a basis for $V^N_{(n)}$. We  choose a basis $\Phi_n$ of $V^\infty_{(n)}$ by picking $M$ large enough so that $V^M_{(n)}$ is saturated, and then taking
\be
\Phi_n := f_M(\Phi^N_n)\ , 
\ee
from which we obtain a homogeneous basis $\Phi$ of $V^\infty$, $\Phi = \bigcup_{n} \Phi_n$.
Now let $a,b,c\in \Phi$. We want to compute the structure constant
\be
C^\infty_{abc}= \lim_{N\to\infty} C^N_{a^Nb^Nc^N}
\ .
\ee
To do this, first write $a^N= f_{NM}(a^M)$. By lemma~\ref{satpro}, for simplicity we can actually choose $a^M \in \cF^M$, that is as a representative of the orbit $\pi_M a^M$. We then have
\begin{multline}\label{3ptsum}
C^N_{a^Nb^Nc^N}
= C^N_{f_{MN}(a^M)f_{MN}(b^M)f_{MN}(c^M)}\\
=\left(\frac{|G_M|^3|\hat G_M^{\cK_{a}}||\hat G_M^{\cK_{b}}||\hat G_M^{\cK_{c}}|}{|G_N|^3|\pstab^{\cK_{a}}||\pstab^{\cK_{b}}||\pstab^{\cK_{c}}|}\right)^{1/2}
\sum_{\vec\sigma \in G_N^{\times 3}} c^N(\sigma_1 g_{MN}(a^M), \sigma_2 g_{MN}(b^M), \sigma_3 g_{MN}(c^M))\ ,
\end{multline}
where $\vec\sigma = (\sigma_1, \sigma_2, \sigma_3)$.
To investigate the limit of  $C^N_{a^Nb^Nc^N}$, we use theorem~2.5 in \cite{Gemunden:2019hie} to rewrite it in a form that makes the $N$ dependence manifest. We can do to this because (\ref{3ptsum}) is essentially their equation (28), the only difference being a prefactor that depends on $M$, but not on $N$, and therefore does not affect convergence as $N\to \infty$.

To do this, let us first introduce some notation. We will write $\cK_1,\cK_2,\cK_3$ for $\cK_a,\cK_b,\cK_c$. We write $\cK_{ij}=\cK_i \cup \cK_j$ and $\cK_{123}=\cK_1\cup\cK_2\cup\cK_3$. Finally we define  the \emph{triple overlap set} $\cK_t = \cK_1 \cap \cK_2\cap \cK_3$ and the \emph{one-point set} $\cK_o = \cK_{123}- \left((\cK_1\cap \cK_2) \cup (\cK_1\cap \cK_3) \cup (\cK_2\cap \cK_3)\right)$. Next we observe that 
\be
c^N(a^N,b^N,c^N)= 0 \qquad \textrm{unless}\qquad \cK_o=\emptyset\ .
\ee
This follows from the fact that $c(a,b,c)=0$ if exactly two of the arguments are in $\C \vac$, as discussed around (\ref{1ptfctvanish}).

Theorem~2.5 of \cite{Gemunden:2019hie} then gives the following expression for the structure constants:
\begin{multline}\label{eq:oligo}
C^N_{a^Nb^Nc^N}
= \left(|G_M|^3|\hat G_M^{\cK_{1}}||\hat G_M^{\cK_{2}}||\hat G_M^{\cK_{3}}|\right)^{1/2}\times \\
\sum_{[\kappa]\in \cS}  M(\kappa\cK,N) \sum_{[\sigma] \in \times_i \Grem{\cK_i}} c^N(\kappa_1\sigma_1 g_{MN}(a^M),\kappa_2\sigma_2 g_{MN}(b^M), \kappa_3\sigma_3 g_{MN}(c^M))\ ,
\end{multline}
where $\sigma = (\sigma_1, \sigma_2, \sigma_3) \in G_N^{\times 3}$, $\kappa = (\kappa_1, \kappa_2, \kappa_3) \in G_N^{\times 3}$ and
\be
\cS= G_N^{\diag} \backslash G_N\times G_N \times G_N/ \sstab^{\cK_1}\times \sstab^{\cK_2} \times \sstab^{\cK_3}
\ee
as a set, where $G_N^{\diag}$ is the diagonal subgroup of $G_N\times G_N\times G_N$. Finally, $\kappa\cK=(\kappa_1\cK_a,\kappa_2\cK_b,\kappa_3\cK_c)$ and
\be\label{MGLmq}
M(\cK,N)  =\left\{ \begin{array}{cc}
\left(\frac{|\pstab^{\cK_1}||\pstab^{\cK_2}||\pstab^{\cK_3}|}{|G_N||\pstab^{\cK_1\cup\cK_2\cup\cK_3}|^2}\right)^{1/2} & \cK_o =\emptyset\\
0& \textrm{else}
\end{array}\right.
\ee
The crucial observation here is that in (\ref{eq:oligo}), only $M(\kappa\cK,N)$ depends on $N$: For $N$ large enough, we can find an $N$-independent representative for $[\kappa]$ and $[\sigma]$. With this choice, the arguments of $c^N$ have $N$-independent support, so that the structure constant $c^N$ of the tensor product VAs does actually  not depend on $N$. For a more detailed explanation of this, see \cite{Gemunden:2019hie}.

\subsection{The VA limit of permutation orbifolds}
From proposition~\ref{permGRDS} and the results in section~\ref{ss:permstruc}, we conclude that the existence of the VA-limit of nested oligomorphic permutation orbifolds only depends on the behavior of $M(\kappa\cK,N)$:
\begin{cor}
Let $V$ be a grading-restricted VA as in (\ref{CFTtype}), and $G_N$ be a nested oligomorphic family. Then the grading-restricted direct system $(V^N,f_{MN})$ defined as in proposition~\ref{permGRDS} has a VA-limit if the $M(\cK,N)$ converge as $N\to \infty$ for all $\cK_1,\cK_2,\cK_3$. 
\end{cor}

The following lemma shows that the $M(\cK,N)$ are actually bounded:
\begin{lem}
\be
0\leq M(\cK,N) \leq 1\ .
\ee
\end{lem}
\begin{proof}
Note that $\hat G^{A\cup B} = \hat G^A \cap \hat G^B$. Denote $G_i = \hat G^{\cK_i}_N$ and $G_{ij}= G_i\cap G_j$ for $i\neq j$ etc. The inequality is trivially satisfied if $\cK_o\neq\emptyset$. If $\cK_o=\emptyset$, then $\cK_1\cup\cK_2\cup\cK_3 = \cK_1\cup\cK_2$ etc., so that $G_{ij}=G_{123}$. We claim that 
\be
|G_1G_2G_3|= \frac{|G_1||G_2||G_3|}{|G_{123}|^2}\ .
\ee
To see this, note that by the usual argument we have $|G_1G_2|=|G_1||G_2|/|G_{12}|$. Next consider the orbit of the set $G_1G_2$ under the right action of $G_3$. The stabilizer subgroup under this action is $G_{123}$: On the one hand, because $G_{123}< G_2$, $G_1G_2g_3 = G_1G_2$ if $g_3\in G_{123}$. On the other hand, if $g_1g_2g_3=\tilde g_1\tilde g_2\tilde g_3$, $\tilde g_3 g_3^{-1}=\tilde g_2^{-1}\tilde g_1^{-1}g_1 g_2$, so that $\tilde g_3 g_3^{-1}$ stabilizes $\cK_1\cap \cK_2$ pointwise; clearly it also stabilizes $\cK_3$ pointwise. But because $\cK_o=\emptyset$, we have $(\cK_1\cap \cK_2) \cup \cK_3=\cK_1\cup\cK_2\cup\cK_3$, so that $\tilde g_3 g_3^{-1}\in G_{123}$. The orbit stabilizer theorem then implies that the orbit has length $|G_3|/|G_{123}|$, from which it follows that $|G_1G_2G_3|=|G_1G_2||G_3|/|G_{123}|=|G_1||G_2||G_3|/|G_{123}|^2$ as claimed. Plugging this into (\ref{MGLmq}) gives
\be\label{MK2}
M(\cK,N) = \left(\frac{|\pstab^{\cK_1}\pstab^{\cK_2}\pstab^{\cK_3}|}{|G_N|}\right)^{1/2}\ ,
\ee
from which the claim follows since the numerator is a subset of the group in the denominator.
\end{proof}
Because the $M(\cK,N)$ are bounded, it is possible to find a VA limit of $V^N$ by picking a convergent subsequence of $V^N$. More precisely, since the basis $\Phi$ is a countable set, so is the set of structure constants $C_{abc}$ with $a,b,c\in \Phi$. We can thus order them, and then, for the first structure constant, pick an infinite subsequence of $\N$ for which all necessary $M(\cK,N)$ converge, giving a limit $C^\infty_{abc}$ for this structure constant. We can apply this procedure recursively to all triples of basis vectors: in the $k$-th step, we keep the first $k$ terms of the ($k$-1)-th subsequence, and then pick an infinite subsequence of the remaining terms for which the $k$-th structure constant  converges. In total this gives a subsequence of $V^N$ for which all structure constants converge, automatically satisfy Borcherds' identity and hence define a state-field map $Y_\infty$. In summary:

\begin{thm}\label{thmoligo}
Let $V$ be a grading-restricted VA as in (\ref{CFTtype}), and $(G_N)_{N\in\N}$ be a nested oligomorphic family of permutation groups. Then we can find a grading-restricted VA $V^\infty$ that is a limit of an appropriate subsequence of the system $(V^N,f_{MN})$ of permutation orbifolds. 
\end{thm}

\section{The large \texorpdfstring{$N$}{N} limit of VOAs}

\subsection{Large central charge limit of Virasoro VOAs}\label{ss:Virasorolimit}
Let us now discuss the large $N$ limit of vertex operator algebras. As the most basic example, let us start out with a sequence of Virasoro VOAs of increasing central charge.

Let $V^N=Vir_{cN}$ be the Virasoro VOA of central charge $cN$ for some $c>1$ with conformal vector $\omega^N$. 
For each $N$, we define a re-scaled copy of the Virasoro algebra by taking $\tilde \omega^N :=\frac{\omega^N}{\sqrt{N}}$ and $Y_N(\tilde\omega^N,z) = \sum_{n\in\Z} \tilde L^N_n z^{-n-2}$ satisfying
\be\label{Virtildecomm}
[\tilde L_m^N,\tilde L_n^N] = \frac1{\sqrt{N}}(m-n)\tilde L^N_{m+n} + \frac c{12}m(m^2-1) \delta_{m,-n}\idV\ .
\ee
Now define
\be
f_{1N}(\vac) = \vac\ , 
\qquad f_{1N}(\omega^1) = \frac{1}{\sqrt{N}}\omega^N = \tilde \omega^N
\ee
and recursively
\be
f_{1N}(L^1_{-n}a) = \frac1{\sqrt{N}}L^N_{-n}f_{1N}(a)\ .
\ee
The maps $f_{1N}$ are clearly bijective, so that we can define connecting maps
\be
f_{MN} = f_{1N}\circ f^{-1}_{1M}\ .
\ee
These clearly satisfy the conditions of definition~\ref{Vdlimit}. 

We claim that the VA-limit is given by the following grading-restricted VA:
Define the Lie algebra
\be\label{Virinfcomm}
[L_m^\infty,L_n^\infty] = \frac c{12}m(m^2-1) \delta_{m,-n}\idV\ ,
\ee
which acts on the graded vector space $Vir_\infty := \mathcal{U}(\mathcal{L}^\infty)\otimes_{\mathcal{U}_{\mathcal{L}^\infty_{(\leq 1)}}}\C\vac$, where as usual $L^\infty_n \vac=0$ for $n\geq -1$. Together with the state-field map $Y(\omega^\infty,z) = \sum_n L^\infty_n z^{-n-2}$ this is then indeed a grading-restricted VA, the grading operator being $L(0)=L^\infty_0$. 

To see that $(Vir_\infty,Y)$ is indeed the VA-limit of the above system, note that $Vir_\infty= \varinjlim  Vir_{cN}$ as a graded vector space. To show that the structure constants $C^N_{abc}$ converge to $C^\infty_{abc}$, proceed as following: Evaluate
\be
b^N_n c^N 
\ee
recursively using Borcherds' identity until it is a linear combination of terms the form $\tilde L^N_{n_1} \cdots \tilde L^N_{n_k}\vac$. Then commute modes $\tilde L^N_n$ with $n\geq-1$ to the right, picking up commutator terms from (\ref{Virtildecomm}). The result is a linear combination of states $a$, from which we can read off the structure constants $C^N_{abc}$.
The point is that this computation differs from the computation of $C^\infty_{abc}$ using (\ref{Virinfcomm}) only by terms of order $O(N^{-1/2})$, so that 
\be
C^\infty_{abc}= C^N_{abc} + O(N^{-1/2})\ ,
\ee
so that $Vir_\infty$ is indeed the VA-limit of $Vir_{cN}$. It is however not a VOA, since (\ref{Virinfcomm}) is not the Virasoro algebra.

\subsection{M\"obius-conformal VAs and Unitary VAs}
The above example shows that the VA-limit of a family of VOAs is in general not a VOA. However, it is not just a grading-restricted VA, but also a M\"obius-conformal VA \cite{MR1651389}. That is, even though it no longer contains a full copy of the Virasoro algebra, it still contains a copy of the global conformal algebra $sl_2(\C)$.

\begin{prop}
The VA-limit $(V^\infty,\vac,Y_\infty)$ of a system of VOAs of CFT type is a grading-restricted M\"obius-conformal VA of CFT type.
\end{prop}
\begin{proof}
Since VOAs are special cases of grading-restricted VAs with $L(0)=L_0$ and $T=L(-1)=L_{-1}$, the only thing left to prove is the existence of the operator $L(1)$. We simply define it as the weak limit of $L_1^N$, $L(1)u:= \lim_{N\to\infty}L_1^N u^N$. Since $L^N_0,L^N_{-1},L^N_{1}$ satisfy the commutation relations of the M\"obius $sl_2(\C)$ Lie algebra, so do $L(0),L(-1),L(1)$.
\end{proof}

Let $V$ be a M\"obius-conformal VA of CFT type. We say $a\in V$ is \emph{quasiprimary} if $L(1)a=0$. $V$ is then spanned by all quasiprimary fields and their $L(-1)$-derivatives (see \eg Remark 4.9d in \cite{MR1651389}).
On $V$, define the bilinear form $B$ from
\be
a_{\wt a+\wt b -1}b  =: B(a,b)\vac\ .
\ee
From skew symmetry it follows that
\be
B(a,b)= (-1)^{\wt a+\wt b}B(b,a)\ .
\ee
We have
\be\label{Bdesc}
B(a,L(-1)b)= L(-1)a_{\wt a+\wt b} b
+ [L(-1),a_{\wt a +\wt b} ] b
= - (\wt a + \wt b) B(a,b)\ ,
\ee
where we used that $a_{\wt a+\wt b}b=0$.
(\ref{Bdesc}) shows that if we know $B$ on the subspace of quasiprimaries, its value on all descendants follows. In particular, if two quasiprimaries are orthogonal to each other, then so all are their descendants.
Finally, if $a,b$ are quasiprimaries, then
$B(a,b)$ vanishes unless $\wt a=\wt b$:
Using the commutation relation
\be
[L(1),a_n]=-(n+2-2\wt a) a_{n+1}+ (L(1)a)_{n+1}\ ,
\ee
we have 
\be
L(1) a_{\wt a+\wt b-2}b = [L(1),a_{\wt a+\wt b-2}]b = -(\wt b -\wt a)a_{\wt a+\wt b-1} b = (\wt a- \wt b)B(a,b)\vac\ .
\ee
The state on the left hand side has weight 0 and is therefore the vacuum. However, since the vacuum is not in the image of $L(1)$, it must vanish. It follows that either $\wt a=\wt b$ or $B(a,b)=0$.
In total we have that $B$ restricted to the subspace of quasiprimaries is blockdiagonal.

Finally, let us say a few words about unitary VAs.
We often want to work with VAs whose bilinear form $B$ is non-degenerate. From (\ref{Bdesc}) it follows that for this it is enough to ensure that $B$ is positive definite on all (finite dimensional) subspaces of quasiprimaries of a given weight. An example of VAs with such a $B$ are unitary VAs:
If the VA $V$ is unitary, then for $a,b$ quasiprimary with $\wt a=\wt b$ the bilinear form $B$ is related to the inner product through
\be
B(a,b) = (\vac, a_{2\wt a-1}b) = (\theta(a),b)\ ,
\ee
where $\theta$ is the anti-linear involution and $(,)$ the positive definite Hermitian form on $V$ \cite{MR3119224}. If we choose a real basis, that is $\theta(a)=a$, then the bilinear form $B$ is given by the same matrix as the inner product, so that it is in particular non-degenerate.

Assume we have a system of unitary grading-restricted VAs $V^N$ with connecting maps $f_{MN}$ that preserve the anti-linear involutions and inner products $\theta_M$ and $(,)_M$, that is $\theta_N\circ f_{MN}=f_{MN}\circ \theta_M$ and $(u,v)_M=(f_{MN}(u),f_{MN}(v))$. Then the VA limit $V^\infty$ is again unitary, with $\theta_\infty(u):=f_N(\theta_N(u^N))$ and $(u,v)_\infty:=(u^N,v^N)_N$. These are clearly again an anti-linear involution and a positive definite Hermitian form.

We note that the connecting maps for permutation orbifolds introduced in definition~\ref{permconnectmap} are compatible with the unitary structure.

\section{Factorization}\label{s:factorize}

\subsection{Factorization in VAs}
The example discussed in section~\ref{ss:Virasorolimit} has another interesting property. Taking a closer look at (\ref{Virinfcomm}), we see that $V^\infty$ is a special kind of VA: it \emph{factorizes}.
\begin{defn}
Let $V$ be a vertex algebra. We say $V$ \emph{factorizes} if there is a set $A\subset V$ of vectors that generate $V$ and that satisfy
\be\label{factorizecommutator}
[a_n,b_m] = D(a,b,n,m)\idV
\ee
for some function $D$ for all $a,b\in A$ and $n,m\in \Z$.
\end{defn}
In physics, such VAs are also called \emph{free field theories}. To avoid confusion with the notion of a free algebra, we use the term `factorization' instead. 

For VAs that factorize, computations become quite simple: Using Borcherds' formula, expressions such as matrix elements or correlation functions can be written in terms of modes of generator fields. These modes can then be commuted through, picking up identity operators only. The matrix elements can thus be obtained by so-called Wick contractions. 

For completeness, let us make this more precise (see for instance also section 3.3 in \cite{MR1651389}).
Using
\be\label{commformula}
[Y(a,z),Y(b,w)]= \sum_{n=0}^\infty Y(a_nb,w) \partial^{(n)}_w \delta(z-w)
\ee
it follows that
\be
[Y(a,z),Y(b,w)]= \sum_{n=0}^\infty \idV D(a,b,n,-1) \partial^{(n)}_w \delta(z-w)
\ee
since for $n\geq0$, $a_nb = a_nb_{-1}\vac =[a_n,b_{-1}]\vac = D(a,b,n,-1)\vac$.

If $V$ is grading-restricted and all generators are homogeneous, we can say something more: (\ref{commformula}) then implies that 
\be
[Y(a,z),Y(b,w)]= \idV B(a,b) \partial^{(\wt(a)+\wt(b)-1)}_w \delta(z-w)\ ,
\ee
which in turn fixes the commutator to be 
\be\label{commcondition}
[a_n,b_m]= B(a,b)\idV \binom{n}{\wt a+\wt b-1} \delta_{n-\wt a+1, -m+\wt b-1}\ .
\ee
Define the annihilation part $Y^+(a,z)$ and the creation part $Y^-(a,z)$ of a field as
\be
Y(a,z) = Y^+(a,z)+Y^-(a,z) = \sum_{n\geq 0}a_n z^{-n-1} + \sum_{n<0}a_n z^{-n-1}\ .
\ee
We then have
\be
[Y^\pm(a,z),Y^\pm(b,w)]=0\ .
\ee
This follows from (\ref{commcondition}) for the $Y^-$ commutator, and from $[a_n,b_m]\vac=0$ for $n,m\geq 0$ for the $Y^+$ commutator.
Finally we have
\bea
{[}Y^+(a,z),Y^-(b,w)]&=& i_{z,w} \frac{B(a,b)}{(z-w)^{\wt a+\wt b}} \idV \\
{[}Y^-(a,z),Y^+(b,w)]&=& -i_{w,z} \frac{B(a,b)}{(z-w)^{\wt a+\wt b}}\idV 
\eea
where $i_{z,w}$ indicates taking the formal power series given by the series expansion of the function for $|z|>|w|$. The functions appearing on the right-hand side are often called \emph{Wick functions} in physics.

We can use these commutators to compute \emph{correlation functions}
\be
\langle \vac, Y(u_1,z_1)Y(u_2,z_2)\ldots Y(u_n,z_n)\vac\rangle\ .
\ee
To do this, we first use  Borcherds' identity to recursively write out $Y(u,z)$ in terms of (residues of) products of generators $Y(a,z)$, giving a correlation function
\be\label{correlator}
\mathcal{C} = \langle \vac, Y(a_1,z_1) \cdots Y(a_m,z_m)\vac\rangle \ .
\ee
We then split all fields into creation and annihilation parts and commute the annihilators to the right, where they annihilate the vacuum, $Y^+(a,z)\vac=0$. This leaves only terms with creation parts and with Wick functions. However, due to the grading, any terms containing creation parts will have a vanishing matrix element when paired with $\vac$. The correlator (\ref{correlator}) is thus simply given by a sum over all possible product Wick functions,
\be
\mathcal{C} = \left\{\begin{array}{cc}\sum_{p \in P^2_{m/2}}\prod_{\{i,j\}\in p} \frac{B(a_i,a_j)}{(z_i-z_j)^{\wt a_i+\wt a_j}} &: m \textrm{\ even}\\
0 &: m  \textrm{\ odd} \end{array}\right. \ .
\ee
Here $P^2_n$ denotes all partitions of the set $\{1,2,\ldots,n\}$ into disjoint pairs, and the product is over all such pairs in the partition $p$. In physics this is called Wick's theorem. It is the analogue of Isserli's theorem in probability theory.

\subsection{The large \texorpdfstring{$N$}{N} limit of symmetric orbifolds}

Before proving the general theorem, let us give one more example of a large $N$ limit that factorizes. Consider symmetric orbifolds, that is permutation orbifolds for which $G_N=S_N$. These were worked out as an example in \cite{Gemunden:2019hie}. 
Picking vectors $v_{1},v_2,v_3$ with $\supp(v_i)=\cK_i$ and $|\cK_i|=K_i$, (\ref{MGLmq}) is given by
\be
 M(\kappa,N) =\left(\frac{(N-K_1)!(N-K_2)!(N-K_3)!}{N!(N-\frac12(K_1+K_2+K_3-n_t(\kappa)))!^2}\right)^{1/2}\ .
\ee
Here we defined  $n_t(\kappa):=|\kappa_1\cK_1\cap\kappa_2\cK_2\cap\kappa_3\cK_3|$ as the length of the triple overlap set under the configuration $\kappa$.
Using Stirling's approximation it follows that for $n_t(\kappa)>0$,
\be\label{Mn3}
 M(\kappa \cK,N) = O(N^{-n_t(\kappa)/2})
\ee
for $N\to\infty$, and for $n_t(\kappa)=0$ 
\be
 M(\kappa \cK,N) \to 1\ . 
\ee
This establishes that only configurations $\kappa$ contribute that have $n_t(\kappa)=0$. Theorem~\ref{mainoligofactor} below will show that therefore symmetric orbifolds indeed factorize in the large $N$ limit. For the moment, we want to use (\ref{Mn3}) to discuss how symmetric orbifolds are generated.

\begin{defn}
Let $v\in V^N:= (V^{\otimes|I_N|})^{G_N}$. We say $v$ is a \emph{single-trace} if $|\supp(v)|=1$. We say $v\in V^\infty$ is single-trace if $v=f_N(v^N)$ for some single-trace vector $v^N$.
\end{defn}

\begin{prop}
Let $V^\infty = \varinjlim  (V^{\otimes N})^{S_N}$ be the limit VA of symmetric orbifolds. Then $V^\infty$ is generated by single-trace vectors.
\end{prop}
\begin{proof}
First let us prove that $V^N$ is generated by single-trace vectors using induction in $n=|\supp(v)|$. The base case $n=1$ is immediate. Let $v$ have $|\supp(v)|$. We can write 
\be
v^N = \pi_N(v^1\otimes v^2 \otimes \cdots \otimes v^n \otimes \vac \otimes \cdots \otimes \vac)
\ee
Next define $u^N=\pi_N(v^n \otimes \vac \cdots \vac)$ and $w^N=\pi_N(v^1\otimes \cdots \otimes v^{n-1}\otimes\vac\cdots \vac$
We then have 
\be
v^N = u^N_{-1}w^N + \ldots
\ee
where the $\ldots$ are vectors of support length $n-1$ or less. By induction, this establishes that $V^N$ is generated by single-trace vectors.

Next consider $v=f_N(v^N) \in V^\infty$. Consider $C^\infty_{xuw}=\langle x,u_{-1}w\rangle$. By (\ref{Mn3}), the structure constant vanishes unless $n_t=0$, which implies that $|\supp(x)|=n$ or $n-2$. The former automatically implies that $x=v$. It follows that 
\be
v = u_{-1} w + \ldots
\ee
where the states in $\ldots$ have support $n-2$ or less. By induction it follows that $V^\infty$ is also generated by single-trace states.

\end{proof}

As a side remark, let us mention that even if the seed VA $V$ is finitely generated, $V^\infty$ is not:
\begin{prop}
Unless $V$ is trivial, $V^\infty = \varinjlim  (V^{\otimes N})^{S_N}$ is not finitely generated.
\end{prop}
\begin{proof}
For any finitely generated VA, the asymptotic growth of $\log \dim V_{(n)}$ for $n\to \infty$ is bounded by $A \sqrt{n}$ for some constant $A$. On the other hand, $\log \dim V^\infty_{(n)}\sim n/\log n$ as $n\to\infty$ (see \eg \cite{Belin:2014fna}).
\end{proof}

\subsection{Factorization for oligomorphic permutation orbifolds}
Now we want to establish under what conditions oligomorphic permutation orbifolds factorize in the large $N$ limit. To this end, we first introduce the following definition:
\begin{defn}
We say a family of permutation groups $( G_N )_{N\in\N}$ \emph{has no finite orbits} if for every finite non-empty set $\cK \subset \N$, the length of the orbit of $\cK$ under $G_N$ diverges, 
\be
O_N(\cK)\to \infty\ .
\ee
\end{defn}

\begin{prop}\label{oligofactor}
Let $G_N$ be nested oligomorphic. Then $M(\cK,N)\to 0$ for all configurations $\cK$ such that $\cK_t\neq\emptyset$ if and only if $G_N$ has no finite orbits. 
\end{prop}
\begin{proof}
Assume $G_N$ has no orbit of finite length. If $\cK_t\neq\emptyset$, then $\pstab^{\cK_1}\pstab^{\cK_2}\pstab^{\cK_3}\subset \pstab^{\cK_t}$. By the orbit-stabilizer theorem we have $O_N(\cK_t)=|G_N|/|\sstab^{\cK_t}|=|G_N|/|\pstab^{\cK_t}||G(\cK_t)|$, where we take $N$ large enough so that condition \ref{condGK} of definition~\ref{conv} applies. From (\ref{MK2}) it follows that
\be
M(\cK,N) \leq |G(\cK_t)|^{-1/2}O_N(\cK_t)^{-1/2} \to 0\ .
\ee
Conversely, let $\cK$ be a set whose orbit length $O_N(\cK)$ is bounded. Consider the configuration $\cK_1=\cK_2=\cK_3=\cK$. Using (\ref{MGLmq}), we have
\be
M(\cK,N) = \left(\frac{|\pstab^{\cK}|}{|G_N|}\right)^{1/2} = |G(\cK_t)|^{-1/2}O_N(\cK)^{-1/2}\ ,
\ee
which does not converge to 0.
\end{proof}

\begin{thm}\label{mainoligofactor}
The large $N$ limit of oligomorphic permutation orbifolds with no finite orbits factorizes.
\end{thm}
\begin{proof}
Clearly $V^\infty$ is generated by states of definite support. Take $u=\pi_N(u^N)$ and $v=\pi^N(v^N)$, where $u^N,v^N$ have support $\cK_u, \cK_v$.
We can specialize Borcherds' identity to obtain the following expression for the commutator (see \eg \cite{MR2023933}):
\be
[u_m,v_n] = \sum_{k\geq 0}\binom{m}{k}(u_kv)_{m+n-k}
\ee
We evaluate the structure constants $C^\infty_{wuv}=\langle w, u_k v\rangle$ by using (\ref{eq:oligo}). We first note that any configuration with $\kappa_3\cK_v \neq \kappa_2\cK_u$ automatically vanishes. This follows because for $k\geq 0$, $u(i)_k\vac=0$ and $(\vac)_k v(i)=0$.
For configurations with $\kappa_3\cK_v = \kappa_2\cK_u$, the structure constant does not vanish only if $\kappa_1\cK_w \subset \kappa_3\cK_v$. However, if $\kappa_1\cK_w\neq \emptyset$, then $\cK_t\neq \emptyset$, so that by proposition~\ref{oligofactor} $C^N_{wuv}\to 0$. It follows that $C^\infty_{wuv}=0$ unless  $\cK_w=\emptyset$, that is $w\in \C\vac$, which implies that indeed only the identity operator appears in the commutator.
\end{proof}

\subsection{Uniqueness of factorizing VAs}
Finally let us briefly discuss uniqueness of VAs that factorize.
Define $F^k$ to be the factorizing grading-restricted M\"obius VA generated by a quasiprimary $v$ of weight $k>0$.  That is, $Y(v,z)= \sum_n v_n z^{-n-1}$ with modes
\be\label{Fkcommutator}
[v_n,v_m]= \idV \binom{n}{2k-1} \delta_{n-k+1,-m+k-1}\ ,
\ee
acting on $\mathcal{U}(\mathcal{V})\otimes_{\mathcal{U}_{\mathcal{V}_{(\leq 1)}}}\C\vac$ , where the vacuum $\vac$ is as usual annihilated by $sl(2)$. Its character is 
\be\label{Fkcharacter}
Z_k(\tau) = \prod_{n\geq k}\frac1{(1-q^n)}\ .
\ee
To see (\ref{Fkcharacter}), we only need to establish that the vectors $v_{-n_1}v_{-n_2}\cdots v_{-n_l}\vac$, $n_i>0$ are linearly independent. But this follows from the fact that their duals maps $\langle \vac, v_{n_1+2k-1}\cdots v_{n_l+2k-1} \cdot\rangle$ are rank 1 and form a dual system to the vectors above, as follows from the commutation relations (\ref{Fkcommutator}).

\begin{prop}
Let $V$ be a grading-restricted M\"obius VA of CFT type with non-degenerate bilinear form $B$. Then $V$ is isomorphic as a VA to
\be
V \cong \bigotimes_{k=1}^\infty \left(F^k\right)^{\otimes N_k}
\ee
for some numbers $N_k \in \N_0$.
\end{prop}
\begin{proof}
For a grading-restricted VA $V$, denote by $V_{(\leq n)}:=\bigoplus_{k\leq n} V_{(k)}$, and denote $U^n:=\bigotimes_{k=1}^{n} \left(F^k\right)^{\otimes N_k}$. By induction in $n$, $V_{(\leq n-1)} \cong U^{n-1}_{(\leq n-1)}$. Denoting
\be
W:=U^{n-1}_{(n)}
\ee
we can use the fact that $B$ is non-degenerate to decompose $V_{(n)}=W\oplus W^\perp$. Note that all vectors in $W^\perp$ are quasi-primary: otherwise $W^\perp$ would contain a descendant of a quasiprimary of lower weight, which would therefore be in $W$ and not in $W^\perp$. We can thus choose a (in general complex) basis $v^i$ of $W^\perp$ such that $B(v^i,v^j)=\delta^{ij}$. This leads to commutators of the form (\ref{Fkcommutator}), and since the $v^i$ are orthogonal to each other, we have $W^\perp = ((F^n)^{\otimes \dim W^\perp})_{(n)}$. Because the $v^i$ have higher weight than all quasiprimaries in $U^{n-1}$, they are orthogonal to them. It follows that $V_{(\leq n)}$ can be written as $\left(\bigotimes_{k=1}^{n} \left(F^k\right)^{\otimes N_k})\right)_{(\leq n)}$ with $N_n=\dim W^\perp$.

\end{proof}

In particular this implies that as long as the VA-limit factorizes and has a non-degenerate bilinear form $B$, then the limit is unique, that is independent of the choice of connecting maps $f_{MN}$. 

Let us summarize the various results that we have found for the physically most relevant case of unitary VOAs of CFT type:
\begin{prop}
Let $V^N$ be a family of unitary VOAs of CFT type together with connecting maps $f_{MN}$ forming a grading-restricted system as in definition~
\ref{Vdlimit} and compatible with the unitary structure. If the structure constants $C^N_{abc}$ converge for all basis vectors, then the VA-limit $(V^\infty,Y_\infty)$ exists and is a grading-restricted unitary M\"obius VA of CFT type. 
Moreover, if this $V^\infty$ factorizes, then the limit is unique up to isomorphism:
that is, if $(V^N,f_{MN})$ and $(V^N,\tilde f_{MN})$ are two systems whose connecting maps $f$ and $\tilde f$ both satisfy the above conditions and whose VA-limits both factorize, then the two limits are isomorphic as grading-restricted VAs.
\end{prop}

\appendix

\bibliographystyle{alpha}
 \bibliography{./refmain}

\end{document}